\documentclass[11pt,letterpaper,english]{article}
\usepackage[margin=1.5in]{geometry}

\usepackage{mdwlist}

\usepackage{amssymb,amsbsy,latexsym}
\usepackage{amsmath}
\usepackage{graphics, subfigure, float}
\usepackage{fp, calc}

\usepackage[latin1]{inputenc}
\usepackage[T1]{fontenc} 
\usepackage{fourier}
\usepackage{bm}

\usepackage{amscd,amsthm}

\usepackage[dvips]{graphicx,epsfig,color}
\usepackage{pst-plot}
\usepackage{pst-all}
\usepackage{pstricks-add,pst-func}
\newpsobject{showgrid}{psgrid}{subgriddiv=1,griddots=10,gridlabels=6pt}
\usepackage{comment,verbatim}

\newtheoremstyle{theorem}{1em}{1em}{\slshape}{0pt}{\bfseries}{.}{ }{}
\theoremstyle{theorem}
\newtheorem{theorem}{Theorem}
\newtheorem*{theorem*}{Theorem}

\newtheorem{lemma}[theorem]{Lemma}
\newtheorem*{lemma*}{Lemma}

\theoremstyle{remark}

\newtheorem*{remark*}{Remark}

\providecommand{\setN}{\mathbb{N}}
\providecommand{\setZ}{\mathbb{Z}}

\providecommand{\setR}{\mathbb{R}}
\providecommand{\setC}{\mathbb{C}}

\newcommand{\E}{\mathop{\mathbb{E}}}

\psset{linewidth=1pt, arrowsize=6pt}
\floatstyle{ruled}
\newfloat{algorithm}{tbp}{loa}
\floatname{algorithm}{Algorithm}

 \theoremstyle{definition}
 
  \theoremstyle{plain}
  
  \theoremstyle{plain}
  
  \theoremstyle{plain}

\theoremstyle{definition}

\newtheorem{openquestion*}{Open Question}
\theoremstyle{theorem}

        \def\drawRect#1#2#3#4#5{
           \FPeval{\x2}{(#2) + (#4)} 
           \FPeval{\y2}{(#3) + (#5)} 
           \pspolygon[#1](#2,#3)(\x2,#3)(\x2,\y2)(#2,\y2)
        }

\psset{nodesep=0.8pt}

\usepackage[displaymath,textmath,sections,graphics, subfigure, floats]{preview} 
\PreviewEnvironment{center} 

\definecolor{augpathgray}{gray}{0.5} 

\usepackage{calrsfs} 
\DeclareMathAlphabet{\pazocal}{OMS}{zplm}{m}{n}

\newcommand{\disc}{\operatorname{disc}}

\makeatother
\date{}
\title{A Fourier-Analytic Approach for the Discrepancy of Random Set Systems} 
\author{Rebecca Hoberg \and Thomas Rothvoss\thanks{University of Washington, Seattle. Email: {\tt rothvoss@uw.edu}. Supported by NSF CAREER grant 1651861 and a David \& Lucile Packard Foundation Fellowship.}}
\begin{document}

\maketitle
%
\begin{abstract}
One of the prominent open problems in combinatorics is the discrepancy of set systems
where each element lies in at most $t$ sets. The Beck-Fiala conjecture suggests that the right bound is $O(\sqrt{t})$, but for three decades the only known bound not depending
on the size of set system has been $O(t)$. Arguably we currently lack
techniques for breaking that barrier. 

In this paper we introduce discrepancy bounds
based on \emph{Fourier analysis}. We demonstrate our method on random set systems. Suppose 
one has  $n$ elements and $m$ sets containing each element independently with 
probability $p$. We prove that in the regime of $n \geq \Theta(m^2\log(m))$, the discrepancy
is at most $1$ with high probability. Previously, a result of Ezra and Lovett gave a bound of $O(1)$ 
under the stricter assumption that $n \gg m^t$. 
\end{abstract}

\section{Introduction}
Let  $([n],\Sigma)$ be a finite set system. For a \emph{coloring}  $\chi:[n]\rightarrow \{-1,1\}$ of the elements, we define the \emph{discrepancy} of the coloring to be the maximum imbalance over all subsets in $\Sigma$. 
The discrepancy of the set system is then defined to be the minimum discrepancy 
over all possible colorings, that is
\[
\disc(\Sigma):=\min_{\chi : [n] \to \{-1,1\}} \max_{S \in \Sigma} \Big|\sum_{i \in S} \chi(i)\Big|. 
\]
Using equivalent matrix notation, we can consider $\bm{A} \in \{ 0,1\}^{m \times n}$  as the \emph{incidence matrix} of the set system, where $m$ is the number of sets. Then a vector $\bm{x} \in \{ -1,1\}^n$ corresponds to a coloring and $\|\bm{A}\bm{x}\|_{\infty}$ is its discrepancy.

One of the seminal results in the field is the theorem of Spencer~\cite{SixStandardDeviationsSuffice-Spencer1985}, which says that the discrepancy of 
a set system is always bounded by $O(\sqrt{n \log(2m/n)})$, assuming that $m \geq n$. The original result was based on the \emph{pigeonhole principle}, going back to work of Beck~\cite{Beck-RothsEstimateIsSharp1981}, and the argument did not provide a polynomial time algorithm to actually find those colorings. A recent line of work~\cite{DiscrepancyMinimization-Bansal-FOCS2010,DiscrepancyMinimization-LovettMekaFOCS12,DiscrepancyForConvexSets-RothvossFOCS2014}, starting with the breakthrough of Bansal, provides algorithms to find colorings that match Spencer's Theorem~\cite{SixStandardDeviationsSuffice-Spencer1985}. All of these algorithms iteratively update a fractional coloring starting at $\bm{0}$ and aim to increase the norm until all elements are colored.

In a setting that has a quite different flavor one assumes that the set system is \emph{sparse} 
in the sense that each element is allowed to be in at most $t$ sets. The Beck-Fiala Theorem~\cite{BECK1981}
shows that the discrepancy is at most $2t-1$, using a linear algebraic approach. 
On the other hand, one
can prove an upper bound $O(\sqrt{t\log(m)})$ using a result by Banaszczyk~\cite{bana98}. In fact, Banaszczyk's Theorem says more generally that for any vectors 
$\bm{v}_1,\ldots,\bm{v}_n \in \setR^m$ of length $\|\bm{v}_i\|_2 \leq \frac{1}{5}$ and any convex body 
$K \subseteq \setR^{m}$ with a Gaussian measure of $\gamma_m(K) \geq 1/2$, there is a coloring 
$\bm{x} \in \{ -1,1\}^n$ so that $\sum_{i=1}^n x_i\bm{v}_i \in K$.  
This result was also non-constructive and based on an operation that deforms the convex set iteratively. Only recently, Bansal, Dadush and Garg~\cite{AlgorithmForKomlosBansalDadushGargFOCS16} found an algorithm matching the $O(\sqrt{t \log(m)})$ bound for coloring $t$-sparse set systems; see also the deterministic approach by Levy et al.~\cite{DetDiscMin-LevyRamadasRothvossIPCO2017}. Even more recently, Bansal, Dadush, Garg and Lovett~\cite{BanaszczykAlgorithmically-BansalDadushGargLovett2017} obtained a polynomial time algorithm that provides the general version of Banaszczyk's Theorem. Their result can be rephrased as follows: given
 any vectors $\bm{v}_1,\ldots,\bm{v}_n \in \setR^m$ with $\|\bm{v}_i\|_2 \leq 1$, one can sample a coloring 
$\bm{x} \in \{ -1,1\}^n$ in polynomial time so that the resulting vector $\sum_{i=1}^n x_i\bm{v}_i$
is \emph{$O(1)$-subgaussian}.
Note that these algorithms still iterately
update a fractional coloring, but additionally make sure that there is ``local progress'' compared to the suffered discrepancy.

Still, if we ask for a bound in the Beck-Fiala setting that only depends on the \emph{frequency} 
parameter $t$, no asymptotic improvement has been made beyond the $2t$ bound of \cite{BECK1981}.
To understand the issue, let us make the additional assumption that all sets have size
at most $O(t)$. Then a folklore argument shows that the discrepancy is bounded by 
$O(\sqrt{t \log(t)})$. To see this, color each element independently at random. Then 
for an individual set, the probability of having discrepancy larger than $O(\sqrt{t \log(t)})$
is bounded by $\frac{1}{\textrm{poly}(t)}$. On the other hand, the dependence degree is at most $O(t^2)$. 
Then the Lov{\'a}sz Local Lemma~\cite{ErdosLovasz3ChromaticHypergraphs1975} implies that there is a positive chance for a good coloring. 
Interestingly, assuming that all sets are large, say bigger than $t^{100}$ does not seem to give 
any advantage. One can use linear algebraic methods to reduce the number of elements to at most the number of sets, but this reduction would destroy the advantage we had in the first place. 
For a more extensive introduction to the field of discrepancy theory we recommend the 
excellent textbooks of Chazelle~\cite{Chazelle2000} and Matousek~\cite{matousek1999}.

This is the initial motivation for us to introduce a very different technique into the field of
discrepancy minimization that is based on \emph{Fourier analysis}. A few years ago Kuperberg, Lovett and Peled~\cite{DBLP:conf/stoc/KuperbergLP12} used a Fourier-analytic approach to show the
existence of \emph{rigid combinatorial structures}. 
For example they can show that there is a 
set $\Pi$ of $|\Pi| \leq O(n^k)$ many permutations on $n$ symbols so that if we sample a permutation $\pi \sim \Pi$ then any $k$-tuple of indices in $\pi$ is distributed as if $\pi$ was a uniform permutation. 
Kuperberg et al.~\cite{DBLP:conf/stoc/KuperbergLP12} achieve this by sampling a large enough set of permutations and then analyzing the Fourier transform.
In fact, Fourier analysis is an often used tool in probability theory. We
would also like to point out the work of Borgs, Chayes and Pittel~\cite{DBLP:journals/rsa/BorgsCP01} who prove that for uniform random integers $a_1,\ldots,a_n \in \{ 1,\ldots,2^{o(n)}\}$, with high probability there is an $\bm{x} \in \{ -1,1\}^n$ so that $|\sum_{i=1}^n a_ix_i| \leq 1$.

We apply our method to the setting of random set systems. The model is as follows: we fix a number $n$ of elements and a number $m$ of sets, where we will assume that $n \gg m$. Then for a probability $p \in [0,\frac{1}{2}]$, we draw a matrix $\bm{A} \in \{ 0,1\}^{m \times n}$ at random by setting each entry $A_{ij}$
to $1$ independently with probability $p$. 
If we later talk about sets and elements, then this refers to the 
set system that has $\bm{A}$ as its incidence matrix. In other words, the sets are $S_1,\ldots,S_m$ and
for an element $j \in [n]$ one has $j \in S_i \Leftrightarrow A_{ij} = 1$. 
We set $t := pm$, as the expected frequency of the elements. Our main result is as follows:
\begin{theorem}\label{thm:mainthm}
Suppose that $n \geq Cm^2\log(m)$ and $t \geq C\log(n)$ where $t := pm$ for $p \in [0,\frac{1}{2}]$ and $C>0$ is a 
large enough constant. Draw 
$\bm{A} \in \{ 0,1\}^{m \times n}$ by letting $\Pr[A_{ij}=1]=p$. Then with high probability there is a vector $\bm{x} \in \{ \pm 1\}^{n}$ so that $\|\bm{A}\bm{x}\|_{\infty} \leq 1$.
\end{theorem} 
Here the phrase ``with high probability'' means with probability $1-\frac{1}{\textrm{poly}(n)}$ where the exponent of the polynomial can be made as large as desired, depending on the constant $C$.
The discrepancy of random set systems has been studied before by 
Ezra and Lovett~\cite{BeckFialaForRandomSetSystems-EzraLovett2016}. Their random model is 
slightly different as for each element they pick exactly $t$ random sets that will contain it.
Either way, for $m \geq n$,  they can show a discrepancy of $O(\sqrt{t\log(t)})$ based on the 
Lov{\'a}sz Local Lemma argument that we mentioned earlier. In the somewhat extreme case of 
$n \gg m^t$ they obtain a discrepancy of $O(1)$. Their argument relies on the observation that in 
this regime, the matrix  $\bm{A}$ will contain every possible column a large constant number of times.

For notation, note that we write all vectors and matrices in bold font.

\section{Overview and Preliminaries}\label{sec:preliminaries}

In the remainder of this paper we study the following random experiment: 
we pick a coloring $\bm{x} \sim \{ -1,1\}^n$ uniformly at random
and let $\bm{D} := \bm{A}\bm{x} \in \setZ^m$ be the random variable that gives the \emph{signed discrepancy}. It would be too naive to hope that $\Pr[\bm{D} = \bm{0}] > 0$ for most matrices $\bm{A}$.
For example if there is even a single set $i$ with an odd number of elements, then $\Pr[\bm{D}=\bm{0}]=0$, so we need to allow some error $\Delta \in \setN$. 
Note that in our setting we will be able to choose $\Delta=1$, but since our framework may apply to settings with larger $\Delta$ we give a more general definition.

Let $\pazocal{R}(\Delta)$ be the distribution of a random variable $R = \sum_{j=1}^{\Delta} r_j$ that is the sum of independent 
random variables $r_j \in \{ -1,0,1\}$ with $\Pr[r_j=1] = \Pr[r_j=-1]=\frac{1}{4}$. By $\pazocal{R}(\Delta)^m$ we denote the distribution of an $m$-dimensional random vector $\bm{R}$ that has 
every coordinate independently drawn from $\pazocal{R}(\Delta)$. To complete our random experiment, 
we draw $\bm{R} \sim \pazocal{R}(\Delta)^m$ and set $\bm{X} := \bm{D} + \bm{R}$.
Then we will prove that $\Pr[\bm{X}=\bm{0}] > 0$, which implies that $\Pr[\|\bm{D}\|_{\infty} \leq \Delta]>0$ as clearly $\bm{R} \in \{ -\Delta,\ldots,\Delta\}^m$.

It appears challenging to show $\Pr[\bm{X}=\bm{0}] > 0$ as the probability in question will be
exponentially small. Similar to Kuperberg et al.~\cite{DBLP:conf/stoc/KuperbergLP12}, this can be done using a 
custom-tailored multi-dimensional \emph{central limit theorem}. 

For a vector-valued random variable $\bm{X} \in \setR^m$, the \emph{Fourier Transform} is
the complex-valued function
\[
 \hat{\bm{X}} : \setR^m \to \setC \quad \textrm{with} \quad  \hat{\bm{X}}(\bm{\theta}) = \E\big[\exp(2\pi i \left<\bm{X},\bm{\theta}\right>)\big] \quad \forall \bm{\theta} \in \setR^m. 
\]
The crucial property of the Fourier coefficients is that they can be used 
to reconstruct the probability of events: 
\begin{lemma}[Fourier Inversion Formula]\label{lem:inversionformula}
For any integer-valued random variable $\bm{X} \in \setZ^m$ and vector $\bm{\lambda} \in \setZ^m$ one has
\[
 \Pr[\bm{X} = \bm{\lambda}] = \int_{[-\frac{1}{2},\frac{1}{2})^m} \hat{\bm{X}}(\bm{\theta}) \cdot \exp\big(-2\pi i \left<\bm{\lambda},\bm{\theta}\right>\big) d\bm{\theta} 
\]
\end{lemma} 
The proof is standard, but for the sake of completeness it can be found in the Appendix. 
As a side remark, note that the coefficients $\hat{\bm{X}}(\bm{\theta})$ are \emph{$\setZ^m$-periodic} and 
instead of integrating over $[-\frac{1}{2},\frac{1}{2}]^m$ one could have integrated over
any set $Q \subseteq \setR^m$ that provides a \emph{tiling} in the sense that $\setZ^m + Q$ partitions the
whole $\setR^m$ apart from measure-0 boundaries.

As we are interested in the case of $\bm{\lambda} = \bm{0}$, the Fourier inversion formula simplifies to
\begin{equation} \label{eq:PrXequalsZeroFTInversion}
  \Pr[\bm{X} = \bm{0}] = \int_{[-\frac{1}{2},\frac{1}{2})^m} \hat{\bm{X}}(\bm{\theta}) d\bm{\theta}
 \stackrel{\textrm{independence}}{=} \int_{[-\frac{1}{2},\frac{1}{2})^m} \hat{\bm{D}}(\bm{\theta}) \cdot \hat{\bm{R}}(\bm{\theta}) d\bm{\theta}.
\end{equation}
Note that $|\hat{\bm{X}}(\bm{\theta})| \leq 1$ for all $\bm{\theta} \in \setR^m$
and since $\bm{X}$ is a \emph{symmetric} random vector, we even have $\hat{\bm{X}}(\bm{\theta}) \in \setR^m$ (the same holds for $\hat{\bm{D}}(\bm{\theta})$ and $\hat{\bm{R}}(\bm{\theta})$). So the challenge is to prove that the positive terms in \eqref{eq:PrXequalsZeroFTInversion} dominate the negative terms. Let $\pazocal{B}^p(\bm{c},r) := \{ \bm{x} \in \setR^m : \|\bm{x}-\bm{c}\|_p \leq r\}$ be the $\ell_p$-ball centered around $\bm{c}$. 
Our analysis works along the following lines: 
\begin{enumerate}
\item[(1)] It is not hard to obtain an explicit expression for the value of $\hat{\bm{D}}(\bm{\theta})$ and with high probability for all $\|\bm{\theta}\|_2 \leq O(\frac{1}{\sqrt{t}})$ that expression can be simplified to 
\[
  \hat{\bm{D}}(\bm{\theta}) = \exp\Big(-2\pi^2 \bm{\theta}^T(\bm{A}\bm{A}^T)\bm{\theta} \pm O(nt^2) \cdot \|\bm{\theta}\|_2^4\Big) 
\] 
In particular the good news is that for $\|\bm{\theta}\|_2 \leq O(\frac{1}{\sqrt{t}})$ one has $\hat{\bm{D}}(\bm{\theta}) > 0$. In fact, a large enough fraction of this positive mass is already 
contained in the significantly smaller ball $\pazocal{B}^2(\bm{0},O(\frac{1}{\sqrt{n}}))$. Integrating gives
\begin{equation} \label{eq:IntegralOverOriginSpike}
 \int_{\|\bm{\theta}\|_2 \leq O(\frac{1}{\sqrt{n}})} \hat{\bm{X}}(\bm{\theta}) d\bm{\theta} \geq \frac{1}{2}\int_{\|\bm{\theta}\|_2 \leq O(\frac{1}{\sqrt{n}})} \hat{\bm{D}}(\bm{\theta}) d\bm{\theta} \geq n^{-c_1m}
\end{equation}
for some constant $c_1>0$. Here we use that $\frac{1}{2} \leq \hat{\bm{R}}(\bm{\theta}) \leq 1$ for all $\|\bm{\theta}\|_2 \leq O(\frac{1}{\sqrt{n}})$
as we will later see.
Note that the very modest positive weight of~\eqref{eq:IntegralOverOriginSpike} has to compensate for all negative contributions elsewhere.
\end{enumerate}

\begin{enumerate}
\item[(2)] A crucial observation is that the quantity $|\hat{\bm{D}}(\bm{\theta})|$ is \emph{$\frac{1}{2}\setZ^m$-periodic}, meaning in particular that for  $\bm{\theta} \in \setR^m$  and $\bm{s} \in \Lambda := \{ -\frac{1}{2},0,\frac{1}{2} \}^m$ we have $|\hat{\bm{D}}(\bm{\theta}+\bm{s})| = |\hat{\bm{D}}(\bm{\theta})|$. While for $\|\bm{\theta}\|_2 \leq O(\frac{1}{\sqrt{t}})$, we know that $\hat{\bm{D}}(\bm{\theta}) > 0$, the values $\hat{\bm{D}}(\bm{\theta}+\bm{s})$ can be either positive or negative if $\bm{s} \in \Lambda \setminus \{ \bm{0}\}$. In fact, it is a good idea to imagine the
\emph{Fourier landscape} as visualized in the figure below with ``spikes'' around all half-integral points.
\begin{center}
\psset{unit=6cm}
\begin{pspicture}(-0.6,-0.7)(0.6,0.6)
\multido{\N=-0.5+0.5}{3}{
\multido{\n=-0.5+0.5}{3}{
\rput[c](\N,\n){\rput[c]{-45}{\psellipticwedge[fillstyle=solid,fillcolor=black!15!white,linewidth=0.25pt](0,0)(0.15,0.1){0}{-1}}}
\rput[c](\N,\n){\rput[c]{-45}{\psellipticwedge[fillstyle=solid,fillcolor=black!30!white,linewidth=0.25pt](0,0)(0.10,0.066){0}{-1}}}
\rput[c](\N,\n){\rput[c]{-45}{\psellipticwedge[fillstyle=solid,fillcolor=black!45!white,linewidth=0.25pt](0,0)(0.05,0.033){0}{-1}}}
}
}
\drawRect{linestyle=dashed}{-0.5}{-0.5}{1}{1}
\psdots[linecolor=darkgray](-0.5,-0.5)(0,-0.5)(0.5,-0.5)(-0.5,0)(0,0)(0.5,0)(-0.5,0.5)(0,0.5)(0.5,0.5)
\cnode*(0,0){2.5pt}{origin} \nput[labelsep=2pt]{0}{origin}{$\bm{0}$}
\rput[c](-0.25,0.55){$[-\frac{1}{2},\frac{1}{2}]^m$}
\pscircle(0,0){0.18}
\cnode*[linecolor=darkgray,fillcolor=lightgray](0.5,0.5){2.5pt}{t} \nput[labelsep=2pt]{0}{t}{$\bm{s} \in \{ -\frac{1}{2},0,\frac{1}{2}\}^m$}
\SpecialCoor
\pnode(0.17;-45){B1}
\pnode(0.28;-60){B2}
\ncline{->}{B2}{B1}
\nput[labelsep=-2pt]{-45}{B2}{$\pazocal{B}^2(\bm{0},O(\sqrt{\tfrac{1}{t}}))$}
\rput[c](0,-0.7){Visualization of $|\hat{\bm{D}}(\bm{\theta})|$}
\end{pspicture}
\end{center}
This is the point where the properties of the additional random term $\bm{R} \sim \pazocal{R}(\Delta)^m$ come into play. First, we are able to show that for $\bm{s}\in \Lambda$ we have
\[\hat{\bm{R}}(\bm{\theta}+\bm{s})\le \hat{\bm{R}}(\bm{\theta})\cdot \prod_{s_i\ne 0}\theta_i^{2\Delta}.\]

Using this we are able to show that if $\Delta\ge 1$ and $\|\bm{\theta}\|_{2} \leq c_2$ for a small enough constant $c_2>0$, then 
\[
  |\hat{\bm{X}}(\bm{\theta})| > 2\sum_{\bm{s} \in \Lambda \setminus \{ \bm{0}\}} |\hat{\bm{X}}(\bm{\theta} + \bm{s})|.
\]
In particular this means the positive spike close to the origin can compensate 
simultaneously for all the potentially negative spikes around the $2^{\Theta(m)}$ points in $\Lambda \setminus \{ \bm{0} \}$.
 \end{enumerate}

\begin{enumerate}
\item[(3)] Finally we need to argue that the coefficients $|\hat{\bm{D}}(\bm{\theta})|$ decay quickly if
$\bm{\theta}$ is far from any half-integral vector. In fact, if $d_2(\bm{\theta},\Lambda) := \min\{ \|\bm{\theta} - \bm{s}\|_2 \mid \bm{s} \in \Lambda\}$ denotes 
the Euclidean distance to $\Lambda$, then one can show that
\[
  \E[|\hat{\bm{D}}(\bm{\theta})|] \leq \exp\left(-c_3 n \cdot \min\left\{ p \cdot d_2(\bm{\theta},\Lambda)^2,1\right\}\right)
\]
where the expectation is over the random choice of the incidence matrix $\bm{A}$. Then with high probability even the integral over all points that are far from any $\{-\frac{1}{2},0,\frac{1}{2}\}$-point is extremely tiny: 
\[
   \int_{\bm{\theta} \in [\frac{1}{2},\frac{1}{2})^m : d_2(\bm{\theta},\Lambda) \geq \Theta(1/\sqrt{t})} |\hat{\bm{D}}(\bm{\theta})| \leq \exp\Big(-c_4 \cdot \frac{n}{m}\Big)
\]
for some $c_4 > 0$.
\end{enumerate}

We will spend the remainder of this paper to fill in the details.


\section{Properties of $R$}\label{sec:propertiesofR}

Recall that we have defined $\pazocal{R}(\Delta)$ as the sum $R = \sum_{j=1}^{\Delta} r_j$
of independent random variables $r_j \in \{ -1,0,1\}$ with $\Pr[r_j=+1]=\frac{1}{4}=\Pr[r_j=-1]$. 
In this paper, we are able to take $\Delta=1$ and so the distribution takes on a very simple form. However, we would like our framework to be useful in other settings where $\Delta \ge 1$ is needed, and so we give the properties of $R$ in this more general setting.

Defined as above, the distribution of $R \sim \pazocal{R}(\Delta)$ is approximately a \emph{discrete Gaussian} with 
variance $\Theta(\Delta)$. We will need a couple of estimates in particular concerning 
the \emph{rate of decay} of $\hat{R}(\theta)$. First, for $\theta \in \setR$ we have
\[
  \hat{R}(\theta) = \E\Big[\exp\Big(2\pi i \cdot \theta \cdot \sum_{j=1}^{\Delta} r_j\Big)\Big] 
\stackrel{\textrm{independence}}{=} \prod_{j=1}^{\Delta} \E\big[\exp(2\pi i \cdot \theta r_j)\big]
= \Big(\frac{1}{2}+\frac{1}{2}\cos(2\pi \cdot \theta)\Big)^{\Delta} 
\]
using that $\exp(i z) = i \sin(z) + \cos(z)$ for all $z \in \setR$.
Using the fact that the Taylor expansion of 
$f(z)=\ln( \frac{1}{2} + \frac{1}{2}\cos(z))$ around $z=0$ is 
$f(z) = -\frac{1}{4}z^2 - \frac{1}{96}z^4 \pm O(z^6)$, 
we obtain
\begin{equation} \label{eq:OneDimRhatUpperBound}
 \hat{R}(\theta) \leq \exp(-\pi^2\cdot \Delta \cdot \theta^2) \quad \forall |\theta| \leq \frac{1}{2}.
\end{equation}
The Taylor expansion
also gives a lower bound of 
\begin{equation} \label{eq:OneDimRhatLowerBound}
\hat{R}(\theta) \geq \exp(-\pi^2\cdot \Delta\cdot \theta^2 - 20 \cdot \Delta \cdot \theta^4) \quad \forall  |\theta| \leq \frac{1}{4}.
\end{equation}
From these formulas we can derive the following:

\begin{lemma} \label{lem:PropertiesOfMultiDimR}
The random vector $\bm{R} \sim \pazocal{R}(\Delta)^m$ has the following properties where $\bm{\theta} \in \setR^m$: 
\begin{enumerate}
\item[(i)] For $\|\bm{\theta}\|_{\infty} \leq \frac{1}{2}$ one has $\hat{\bm{R}}(\bm{\theta}) \leq \exp(-\pi^2 \Delta \cdot \|\bm{\theta}\|_2^2)$.
\item[(ii)] For $\|\bm{\theta}\|_{\infty} \leq \frac{1}{4}$ one has $\hat{\bm{R}}(\bm{\theta}) \geq \exp(-\pi^2  \Delta \cdot \|\bm{\theta}\|_2^2 - 20 \Delta \cdot \|\bm{\theta}\|_2^4)$.
\item[(iii)] For $\|\bm{\theta}\|_{\infty} \leq \frac{1}{8}$ and $\bm{s} \in \{ -\frac{1}{2},0,\frac{1}{2} \}^m$, one has 
\[
  \frac{\hat{\bm{R}}(\bm{\theta}+\bm{s})}{\hat{\bm{R}}(\bm{\theta})} \leq 
 \prod_{i \in \textrm{supp}(\bm{s})} \big(32\theta_i^2\big)^{\Delta}
\]
\end{enumerate}
\end{lemma}
\begin{proof}
Using \eqref{eq:OneDimRhatUpperBound} and the fact that the coordinates of $\bm{R}$ are chosen independently we get
\[
  \hat{\bm{R}}(\bm{\theta}) = \E\Big[\exp\Big(2\pi i \sum_{i=1}^m \theta_iR_i\Big)\Big] 
\stackrel{\textrm{independence}}{=} \prod_{i=1}^m \hat{R}_i(\theta_i)
\stackrel{\eqref{eq:OneDimRhatUpperBound}}{\leq} \exp\left(-\pi^2 \Delta \cdot \|\bm{\theta}\|_2^2\right)
\]
for $\|\bm{\theta}\|_{\infty} \leq \frac{1}{2}$, where $R_i \sim \pazocal{R}(\Delta)$ is the one 1-dimensional distribution.
The lower bound in $(ii)$ follows along the same lines using \eqref{eq:OneDimRhatLowerBound} instead. 

To show $(iii)$, consider the function $g(z) := \frac{1}{2}+\frac{1}{2}\cos(2\pi z)$.
In particular for all $|z| \leq \frac{1}{8}$ one has $g(z) \geq 1-16z^2 \geq \frac{1}{2}$
and $g(\frac{1}{2}+z) \leq 16z^2$ as well as $g(-\frac{1}{2}+z) \leq 16z^2$. Then
\[
  \frac{\hat{\bm{R}}(\bm{\theta}+\bm{s})}{\hat{\bm{R}}(\bm{\theta})} = \prod_{i \in \textrm{supp}(\bm{s})} \frac{g(s_i+\theta_i)^{\Delta}}{g(s_i)^{\Delta}} \leq \prod_{i \in \textrm{supp}(\bm{s})}\big(32\theta_i^2\big)^{\Delta}
  \]
  
\end{proof}

\section{The Fourier transform close to the origin}\label{sec:proofoverview}


In this section we work toward estimating the integral $\int_{\|\bm{\theta}\|_2 \leq r} \hat{\bm{D}}(\bm{\theta}) d\bm{\theta}$ for suitable small radius $r$. 
We begin with obtaining an explicit formula for the Fourier coefficients:
\begin{lemma}\label{lem:fourierformulaForD}
For any $\bm{\theta} \in \setR^m$ one has 
\[
\hat{\bm{D}}(\bm{\theta}) 
= \prod_{j=1}^n \cos\big(2\pi \left<\bm{A}^j,\bm{\theta}\right>\big) 
\]
\end{lemma}
\begin{proof}
We can write
\begin{eqnarray*}
 \hat{\bm{D}}(\bm{\theta}) &=&  \E\big[\exp(2\pi i \left<\bm{D},\bm{\theta}\right>)\big] = \E\Big[\exp\Big(2\pi i\Big<\sum_{j=1}^n x_j\bm{A}^j,\bm{\theta}\Big>\Big)\Big] \\
&\stackrel{\textrm{indep.}}{=}& \prod_{j=1}^n \E_{x_j \sim \{\pm 1\}}\big[\exp\big(2\pi i \cdot \left<\bm{A}^j,\bm{\theta}\right> \cdot x_j\big)\big] \\
&=& \prod_{j=1}^n \Big( \frac{1}{2} \exp\big(2\pi i \left<\bm{A}^j,\bm{\theta}\right>\big) + \frac{1}{2} \exp\big(-2\pi i \left<\bm{A}^j,\bm{\theta}\right>\big) \Big) \\
&=& \prod_{j=1}^n \cos\big(2\pi \left<\bm{A}^j,\bm{\theta}\right>\big)
\end{eqnarray*}
where we use again the elementary fact that $\exp(iz) = \cos(z) + i \cdot \sin(z)$ for $z \in \setR$.
\end{proof}
Notice that if $\bm{\theta}$ is too large, then it is possible that $|\left<\bm{A}^j,\bm{\theta}\right>| \geq \frac{1}{4}$ and the factor $\cos(2\pi \left<\bm{A}^j,\bm{\theta}\right>)$ might be negative. But for $\|\bm{\theta}\|_2 \leq O(\frac{1}{\sqrt{t}})$, we can show that $\hat{\bm{D}}(\bm{\theta})$ is positive and we will give a good estimate for it.

Let $\bm{I}_{m} \in \setR^{m \times m}$ be the identity matrix. We abbreviate 
\[
  \bm{\Sigma}[\bm{D}] := \bm{A}\bm{A}^T, \quad \bm{\Sigma}[\bm{R}] := \frac{\Delta}{2} \cdot \bm{I}_m \quad \textrm{and} \quad \bm{\Sigma}[\bm{X}] := \bm{\Sigma}[\bm{D}] + \bm{\Sigma}[\bm{R}].
\]
As we will see, these are the $m \times m$ \emph{covariance matrices} of $\bm{D}$, $\bm{R}$ and $\bm{X}$.
Note that $\bm{\Sigma}[\bm{D}]_{i,i'} = \left<\bm{A}_i,\bm{A}_{i'}\right> = |S_i \cap S_{i'}|$ for $i,i' \in [m]$. In particular $\E[\bm{\Sigma}[\bm{D}]_{i,i}] = \E[|S_i|] = np $ 
and for $i \neq i'$, $\E[\bm{\Sigma}[\bm{D}]_{i,i'}] = np^2$. 
Then coordinate-wise
\begin{equation}\label{eqn:CovarianceD}
 \E[\bm{\Sigma}[\bm{D}]] = (1-p) pn \cdot \bm{I}_m + p^2n \cdot \bm{1}\bm{1}^T
\end{equation}
where $\bm{1}\bm{1}^T \in \setR^{m \times m}$ is the rank-1 all-ones matrix.

\begin{lemma}\label{lem:DapproxGaussian}
With high probability over the choice of $\bm{A}$, for all $\|\bm{\theta}\|_2\le \frac{1}{16\sqrt{t}}$ we have\footnote{By $\hat{\bm{D}}(\bm{\theta}) = \exp(a \pm b)$ what we mean precisely is that there is $\beta(\bm{\theta}) \in [-b,b]$ with $\hat{\bm{D}}(\bm{\theta}) = \exp(a + \beta(\bm{\theta}))$.}
\[
 \hat{\bm{D}}(\bm{\theta}) = \exp\Big(-2\pi^2 \bm{\theta}^T\bm{\Sigma}[\bm{D}]\bm{\theta} \pm O(nt^2\|\bm{\theta}\|_2^4) \Big)
\]
\end{lemma}

\begin{proof}
Since $t \geq C\log(n)$, we know via a standard Chernov bound argument that no element
will be in more than $4t$ sets, which means that $\|\bm{A}^j\|_2\le 2\sqrt{t}$ for 
each column $j \in [n]$. Then for any $\|\bm{\theta}\|_2 \leq \frac{1}{16\sqrt{t}}$
we can use the Cauchy-Schwarz inequality to get  $|\langle \bm{A}^j,\bm{\theta}\rangle|\le \|\bm{A}^j\|_2 \cdot \|\bm{\theta}\|_2 \leq \frac{1}{8}$.  
Using a similar Taylor expansion as before one can show that 
$\exp(-\frac{1}{2}z^2-z^4) \leq \cos(z) \leq \exp(-\frac{1}{2}z^2+z^4)$ for all $|z| \leq \frac{1}{8}$.
Therefore we can express
\begin{eqnarray*}
\hat{\bm{D}}(\bm{\theta}) &\stackrel{\textrm{Lem.\ref{lem:fourierformulaForD}}}{=}& \prod_{j=1}^n \cos(2\pi\left<\bm{A}^j,\bm{\theta}\right>)  \\
 &=& \exp\Big(-2\pi^2 \sum_{j=1}^n\left<\bm{A}^j,\bm{\theta}\right>^2 \pm (2\pi)^4 \sum_{j=1}^n \left<\bm{A}^j,\bm{\theta}\right>^4\Big) \\
&=& \exp\Big(-2\pi^2 \bm{\theta}^T\underbrace{(\bm{A}\bm{A}^T)}_{=\bm{\Sigma}[\bm{D}]}\bm{\theta} \pm O(1)  nt^2\|\bm{\theta}\|_2^4\Big) 
\end{eqnarray*}
as claimed. Here we use in the last step that $\left<\bm{A}^j,\bm{\theta}\right>^4 \leq \|\bm{A}^j\|_2^4 \cdot \|\bm{\theta}\|_2^4 \leq (2\sqrt{t})^4 \cdot \|\bm{\theta}\|_2^4$.
\end{proof}

\subsection{Comparison with a Gaussian distribution}

It will be instructive to compare $\hat{\bm{D}}(\bm{\theta})$ to the Fourier transform 
of an appropriately scaled \emph{Gaussian} and separately obtain an integral for the Gaussian.
So, consider an $m$-dimensional Gaussian $\bm{Y}$ with expectation $\bm{0}$ and covariance matrix $\bm{\bm{\Sigma}}[\bm{Y}]$. 
Then the \emph{density} function is well-known to be
\[
 f_{\bm{Y}}(\bm{x}) = \frac{1}{(2\pi)^{m/2}\sqrt{\det(\bm{\Sigma}[\bm{Y}])}} \exp\Big(-\frac{1}{2} \bm{x}^T \bm{\Sigma}[\bm{Y}]^{-1}\bm{x}\Big)
\] 
and the Fourier transform is 
$\hat{\bm{Y}}(\bm{\theta}) = \exp(-2\pi^2 \bm{\theta}^T\bm{\Sigma}[\bm{Y}]\bm{\theta})$.

\begin{lemma}\label{lem:gaussianbound} Consider a Gaussian $\bm{Y}$ with covariance matrix $\bm{\Sigma}[\bm{Y}]$. 
Then for $\lambda \geq 0$,
\[ \int_{\bm{\theta}: \bm{\theta}^T \bm{\Sigma}[\bm{Y}] \bm{\theta} \leq \frac{1}{4\pi^2}(\sqrt{m}+\lambda)^2} \hat{\bm{Y}}(\bm{\theta}) d\bm{\theta} \ge f_{\bm{Y}}(\bm{0}) \cdot (1-2e^{-\lambda^2/2}).\]
In particular, if $\bm{\Sigma}[\bm{Y}]=r\bm{I}_m$, then $\int_{\|\bm{\theta}\|_2\leq \frac{1}{\pi} \sqrt{m/r}} \hat{\bm{Y}}(\bm{\theta}) d\bm{\theta} \geq \frac{1}{2} \cdot (2\pi r)^{-m/2}$.
\end{lemma}

\begin{proof}
Let $\bm{G} \sim N^m(0,1)$ be the standard Gaussian with covariance matrix $\bm{I}_m$. 
Then $\E[\|\bm{G}\|_2^2] = m$ and by Jensen's inequality $\E[\|\bm{G}\|_2] \leq \sqrt{m}$.
Observe that $\| \cdot \|_2$ is a \emph{1-Lipschitz} function, so in particular $\Pr[\|\bm{G}\|_2 > \sqrt{m} + \lambda] \leq 2e^{-\lambda^2/2}$ using the inequality of Sudakov-Tsirelson\footnote{Arguably, this is overkill. For our purpose it would also suffice to apply Markov's inequality to get $\Pr[\|\bm{G}\|_2 > 2\sqrt{m}] \leq \frac{1}{2}$.}. 
Then we estimate
\begin{eqnarray*}
 & & \int_{\bm{\theta}: \bm{\theta}^T \bm{\Sigma}[\bm{Y}] \bm{\theta} \leq \frac{1}{4\pi^2}(\sqrt{m}+\lambda)^2} \hat{\bm{Y}}(\bm{\theta}) d\bm{\theta} \\
 &=& \int_{\bm{\theta}: \|2\pi \bm{\Sigma}[\bm{Y}]^{1/2}\bm{\theta}\|_2 \leq \sqrt{m} + \lambda} \exp\Big(-\frac{1}{2} \cdot \|2\pi\bm{\Sigma}[\bm{Y}]^{1/2}\bm{\theta}\|_2^2\Big)d\bm{\theta} \\
 &\stackrel{\textrm{transformation}}{=}& \frac{(2\pi)^{m/2}}{\det(2\pi \bm{\Sigma}[\bm{Y}]^{1/2})} \frac{1}{(2\pi)^{m/2}} \int_{\bm{z}: \|\bm{z}\|_2 \leq \sqrt{m}+\lambda} \exp\Big(-\frac{1}{2} \|\bm{z}\|_2^2 \Big) d\bm{z} \\
 &=& \frac{1}{(2\pi)^{m/2} \sqrt{\det(\bm{\Sigma}[\bm{Y}])}} \Pr_{\bm{G} \sim N^m(0,1)}[\|\bm{G}\|_2 \leq \sqrt{m} + \lambda] \\
&\geq& f_{\bm{Y}}(\bm{0}) \cdot (1-2e^{-\lambda^2/2})
\end{eqnarray*}
Here we use an integral transformation in the form $\int_{\bm{x} \in \setR^m} f(\bm{B}\bm{x}) d\bm{x} = \frac{1}{\det(\bm{B})} \int_{\bm{x} \in \setR^m} f(\bm{x}) d\bm{x}$.
For the second claim we set $\lambda := \sqrt{m}$. Then for $\bm{\Sigma}[\bm{Y}] = r \cdot \bm{I}_m$ one has 
$\bm{\theta}^T\bm{\Sigma}[\bm{Y}]\bm{\theta} \leq \frac{1}{4\pi^2} (2\sqrt{m})^2 \Leftrightarrow \|\bm{\theta}\|_2 \leq \frac{1}{\pi} \cdot \sqrt{\frac{m}{r}}$ and moreover $f_Y(\bm{0}) = \frac{1}{(2\pi)^{m/2}} \frac{1}{\sqrt{\det(r \cdot \bm{I}_m)}} = (2\pi r)^{-m/2}$.
\end{proof}

\begin{lemma}\label{lem:lowerbound}
With high probability over the choice of $\bm{A}$, we have
$$\int_{\|\bm{\theta}\|_2\le \frac{1}{16\sqrt{t}}}\hat{\bm{X}}(\bm{\theta})d\bm{\theta} \ge n^{-\Theta(m)}.$$
\end{lemma}

\begin{proof} 
We will actually integrate over a radius that is quite a bit smaller than $\frac{1}{16\sqrt{t}}$, but note that $\hat{\bm{X}}(\bm{\theta}) > 0$ for all $\|\bm{\theta}\|_2 \leq \frac{1}{16\sqrt{t}}$. In fact, 
we will lower bound this integral by comparison with the Gaussian with covariance matrix
$r \cdot \bm{I}_m$ for parameter $r := mn$. Then by Lemma~\ref{lem:gaussianbound} it suffices to integrate
over $\bm{\theta}$'s with $\|\bm{\theta}\|_2 \leq \frac{1}{\pi} \sqrt{\frac{m}{r}} = \frac{1}{\pi} \frac{1}{\sqrt{n}}$. 
Note that even deterministically 
$\bm{\theta}^T\bm{\Sigma}[\bm{D}]\bm{\theta} \leq \frac{1}{2}nm \cdot \|\bm{\theta}\|_2^2$ for all $\bm{\theta} \in \setR^m$ (using (\ref{eqn:CovarianceD}) with $p\le \frac{1}{2}$).
With the assumptions $\|\bm{\theta}\|_2 \leq \frac{1}{2\pi\sqrt{n}}$ and $\|\bm{A}^j\|_2 \leq 2\sqrt{t}$, and $\Delta\leq \frac{n}{8}$
we can write
\begin{eqnarray*}
  \hat{\bm{X}}(\bm{\theta}) &=& \hat{\bm{R}}(\bm{\theta}) \cdot \hat{\bm{D}}(\bm{\theta}) \geq \underbrace{\exp\Big(-(\pi^2+20) \Delta \|\bm{\theta}\|_2^2\Big)}_{\geq 1/2} \cdot \exp\Big(-2\pi^2 \underbrace{ \bm{\theta}^T\bm{\Sigma}[\bm{D}]\bm{\theta}}_{\leq nm \|\bm{\theta}\|_2^2/2} - ct^2n \underbrace{\|\bm{\theta}\|_2^2}_{\leq 1/n}\|\bm{\theta}\|_2^2 \Big) \\
 &\geq& \frac{1}{2} \exp\Big(- 2\pi^2 \bm{\theta}^T\Big( \underbrace{(\frac{1}{2}mn+\frac{ct^2}{2\pi^2})}_{ \leq mn} \bm{I}_m\Big) \bm{\theta} \Big) \geq \frac{1}{2} \hat{\bm{Y}}(\bm{\theta}).
\end{eqnarray*}
Here we use the lower bound on $\hat{\bm{R}}(\bm{\theta})$ from Lemma~\ref{lem:PropertiesOfMultiDimR}(ii)
and the estimate on $\hat{\bm{D}}(\bm{\theta})$ from Lemma~\ref{lem:DapproxGaussian}.
By Lemma \ref{lem:gaussianbound}, we can then simply use that
\[
\int_{\|\bm{\theta}\|_2\le \frac{1}{\pi \sqrt{n}}}\hat{\bm{Y}}(\bm{\theta})d\bm{\theta}\ge (2\pi nm)^{-m/2}.
\]
\end{proof}
The reader might have observed that the bound we used was terribly wasteful. Effectively we 
have upper bounded $\bm{\Sigma}[\bm{X}] \preceq nm \cdot \bm{I}_m$. 
Instead one could have 
used powerful \emph{matrix Chernov bounds} and argue that $\bm{\Sigma}[\bm{X}] \preceq (1 + \varepsilon) \E[\bm{\Sigma}[\bm{X}]]$ as long as $t \gg \frac{\log(n)}{\varepsilon^2}$. 
Then one could have done a more careful integration. However, even such a careful treatment would only
affect the constant in the exponent of the right hand side quantity $n^{-\Theta(m)}$. For that reason we skip such a more careful estimate.

\section{Dominance of the central spike}

Recall that for half-integral points $\bm{s} \in \Lambda := \{ -\frac{1}{2},0,\frac{1}{2} \}^m$ and $\bm{\theta} \approx \bm{0}$ we have
$|\hat{\bm{D}}(\bm{s}+\bm{\theta})| \approx 1$ while we lack control of the sign for all points $\bm{s} \neq \bm{0}$.
But we know that $\hat{\bm{D}}(\bm{s} + \bm{\theta}) = \hat{\bm{D}}(\bm{\theta})$ and $\hat{\bm{X}}(\bm{\theta}) = \hat{\bm{D}}(\bm{\theta}) \cdot \hat{\bm{R}}(\bm{\theta})$. So the
 crucial argument will be that, if weighted with the fast decaying Fourier 
coefficient $\hat{\bm{R}}(\bm{\theta})$, the ``central spike'' around $\bm{0}$
has a larger contribution than all the $3^{m}-1$ other spikes. 

\begin{lemma} 
For a small enough constant $c>0$, let $\Delta =1$ and draw $\bm{R} \sim \pazocal{R}(\Delta)^m$. 
Then for $\|\bm{\theta}\|_{2} \leq c$ one has  \[|\hat{\bm{R}}(\bm{\theta})|>2\sum_{\bm{s}\in \{-\frac{1}{2},0,\frac{1}{2}\}^m\setminus\{\bm{0}\}}|\hat{\bm{R}}(\bm{\theta}+\bm{s})|.\]
\end{lemma}

\begin{proof}

Using Lemma~\ref{lem:PropertiesOfMultiDimR} (iii) we obtain
 \begin{eqnarray*}
 \sum_{\bm{s}\in \Lambda\setminus\{\bm{0}\}}|\hat{\bm{R}}(\bm{\theta}+\bm{s})| 
 &\le& \hat{\bm{R}}(\bm{\theta})\cdot \sum_{\bm{s}\in\Lambda\setminus\{\bm{0}\}} \Big(\prod_{s_i\ne 0} 32\theta_i^2\Big) \\
 &=& \hat{\bm{R}}(\bm{\theta})\cdot \Big(\prod_{i=1}^m(64\theta_i^2+1)-1\Big) \\
 &\le& \hat{\bm{R}}(\bm{\theta})\cdot (e^{64\|\bm{\theta}\|_2^2}-1) \stackrel{\|\bm{\theta}\|_2^2 \leq \frac{1}{64}}{\leq} \hat{\bm{R}}(\bm{\theta}) \cdot 128\|\bm{\theta}\|_2^2.
 \end{eqnarray*}
Here the equality follows from expanding the product $\prod_{i=1}^m(32\theta_i^2+32\theta_i^2+1)$, which gives a corresponding term for every possible $\bm{s}$. For $\|\bm{\theta}\|_2$ small enough, the claim holds.

 
%
\end{proof}

Since $\bm{D}$ is $\Lambda$-periodic, this implies the same relation for the Fourier coefficients of $\bm{X} = \bm{D} + \bm{R}$:
\begin{lemma}\label{lem:halfintegralbound}
For $\Delta=1$ and $\|\bm{\theta}\|_{2} \leq c$ for a small enough $c>0$ one has
\[
   |\hat{\bm{X}}(\bm{\theta})| > 2\sum_{\bm{s} \in \Lambda \setminus \{\bm{0}\}} |\hat{\bm{X}}(\bm{\theta}+\bm{s})|.
\]
\end{lemma}
\begin{proof}
Use that $\hat{\bm{X}}(\bm{\theta}) = \hat{\bm{D}}(\bm{\theta}) \cdot \hat{\bm{R}}(\bm{\theta})$ and
$|\hat{\bm{D}}(\bm{\theta}+\bm{s})| = |\hat{\bm{D}}(\bm{\theta})|$ for all $\bm{s} \in \Lambda$.
\end{proof}


\section{Bounding the Fourier transform far from any half-integral point}

Finally, we will show that with high probability over the random choice of $\bm{A}$, 
the Fourier coefficients $|\hat{\bm{D}}(\bm{\theta})|$ decay very quickly as we move away from 
half integral points. 
We define the $\ell_p$ distance to $\Lambda$ as $d_p(\bm{\theta},\Lambda)=\min_{\bm{z}\in \Lambda}\|\bm{\theta}-\bm{z}\|_p$, where again 
 $\Lambda=\{-\frac{1}{2},0,\frac{1}{2}\}^m$. We will show the following:
\begin{lemma}\label{lem:intDhatupperbd}
With probability at least $1-\exp\left(-\Theta(n/m)\right)$ (over the random choice of $\bm{A}$) one has
\[\int_{\bm{\theta} \in [-\frac{1}{2},\frac{1}{2})^m: d_2(\bm{\theta},\Lambda)\ge \frac{1}{16\sqrt{t}}} |\hat{\bm{D}}(\bm{\theta})| d\bm{\theta} \leq \exp(-\Theta(n/m)).\]
\end{lemma}

Recall that 
\[
  \hat{\bm{D}}(\bm{\theta}) = \prod_{j=1}^n \cos\Big(2\pi \sum_{i=1}^m A_{ij}\theta_i\Big)
\]
for all $\bm{\theta} \in \setR^m$ and $\bm{A} \in \{ 0,1\}^{m \times n}$ is a random matrix with $\Pr[A_{ij}=1]=p$
for each entry independently. We want to show that $\hat{\bm{D}}(\bm{\theta})$ is decaying for all points that are far from 
half-integral points. But by periodicity of the cosine, $|\hat{\bm{D}}(\bm{\theta})| = |\hat{\bm{D}}(\bm{\theta} + \bm{s})|$ for any $\bm{s}\in \{-\frac{1}{2},0,\frac{1}{2}\}^m$, and so it suffices to consider the 
points with $\|\bm{\theta}\|_{\infty} \leq \frac{1}{4}$. 
We begin with a bound that is useful if $\bm{\theta}$ has a large entry and analyze the contribution of each of the $n$
factors separately:
\begin{lemma} \label{lem:BoundingOneFourierFactorForLargeThetaEntries}
Let $\bm{\theta} \in \setR^m$ with $\|\bm{\theta}\|_{\infty} \leq \frac{1}{4}$.
Draw $\bm{a} \in \{ 0,1\}^m$ with $\Pr[a_i=1]=p$ independently. Then $\E[|\cos(2\pi \sum_{i=1}^m a_i\theta_i)|] \leq 1 - \frac{\pi^2}{4} \cdot p \cdot \|\bm{\theta}\|_{\infty}^2$.
\end{lemma}
\begin{proof}
W.l.o.g. $\theta_1 = \|\bm{\theta}\|_{\infty} \geq 0$. Fix any outcome for 
$a_2,\ldots,a_m$ and let $s := 2\pi\sum_{i=2}^m a_i\theta_i$. Note that $|2\pi \theta_1| \leq \frac{\pi}{2}$. Then
at least one of the outcomes $\{s,s+2\pi \theta_1\}$ has a distance 
of at least $\pi \theta_1$ to the nearest multiple of $\pi$. 
Pessimistically that outcome is attained with probability $p$. Hence  
\[
  \E_{a_1}[|\cos(s + 2\pi a_1\theta_1)| ] \leq (1-p) \cdot \cos(0) + p \cdot |\cos(\underbrace{\pi \theta_1}_{\leq \pi/2})|
\leq (1-p) + p \cdot \Big(1 - \frac{1}{4} \cdot (\pi \theta_1)^2\Big) = 1-\frac{\pi^2}{4}p\theta_1^2
\]
using that $|\cos(x)| \leq 1-\frac{1}{4}x^2$ for $|x| \leq \frac{\pi}{2}$.
\end{proof}

The orthogonal case that we need to analyze is the following: 
\begin{lemma} \label{lem:ExpectedCosineForSmallL2NormCoefficients}
For a small enough constant $b>0$ the following holds: 
Suppose that $y_1,\ldots,y_m \in \setR$ are independent random variables with $|y_i| \leq 1, \E[y_i]=0$
and $\frac{p}{2} \leq \E[y_i^2] \leq p$ with $0 \leq p \leq \frac{1}{2}$. 
Suppose that $\bm{\theta} \in \setR^m$ is a vector with $\|\bm{\theta}\|_{\infty} \leq \frac{1}{4}$, $p\|\bm{\theta}\|_2^2 \leq b$.
Then for all $s \in \setR$ one has 
\[
  \E\Big[\Big|\cos\Big(s+2\pi\sum_{i=1}^m \theta_iy_i\Big)\Big|\Big] \leq 1-\frac{1}{2} \cdot p\cdot \|\bm{\theta}\|_2^2
\] 
\end{lemma}

\begin{proof}
Since $|\cos(x)|$ is $\pi$-periodic, we may assume that 
$-\frac{\pi}{2} \leq s \leq \frac{\pi}{2}$ and hence $0 \leq \cos(s) \leq 1$.
If $-\frac{\pi}{2} \leq s < 0$, we can also replace $y_i$ with $-y_i$ and $s$ with $-s$ without affecting the claim and hence assume
that $0 \leq s \leq \frac{\pi}{2}$. 

Consider the random variable $Z := \sum_{i=1}^m \theta_iy_i$. First of all, we have 
$\E[Z^2] = \textrm{Var}[Z] = \sum_{i=1}^m \theta_i^2 \E[y_i^2]$. The 4th moment of $Z$ can be written as: 
\begin{eqnarray*}
\E[Z^4] &=& \E\Big[\Big(\sum_{i=1}^m \theta_iy_i\Big)^4\Big] 
= \sum_{i_1=1}^m \sum_{i_2=1}^m \sum_{i_3=1}^m \sum_{i_4=1}^m \theta_{i_1}\theta_{i_2}\theta_{i_3}\theta_{i_4}\E[y_{i_1}y_{i_2}y_{i_3}y_{i_4}] \quad \quad (*) \\
&=& \sum_{i=1}^m \theta_i^4 \E[y_i^4] + \frac{1}{2} {4 \choose 2} \sum_{i \neq j} \theta_i^2 \theta_j^2 \E[y_i^2]\cdot \E[y_j^2] \\
&\stackrel{|y_i| \leq 1}{\leq}& \|\bm{\theta}\|_{\infty}^2 \underbrace{\sum_{i=1}^m \theta_i^2\E[y_i^2]}_{=\E[Z^2]} + 3 \cdot \Big(\underbrace{\sum_{i=1}^m \theta_i^2 \E[y_i^2]}_{=\E[Z^2]}\Big)^2 
\leq \|\bm{\theta}\|_{\infty}^2 \cdot \E[Z^2] + 3 \cdot \E[Z^2]^2.
\end{eqnarray*}
Note that the expectations $\E[y_{i_1}y_{i_2}y_{i_3}y_{i_4}]$ are $0$ unless $i_1=i_2=i_3=i_4$ or there are two identical index pairs, 
for example with $i_1=i_2$ and $i_{3}=i_4$. We distinguish two further cases for the regime of $s$: 
\begin{itemize}
\item \emph{Case $0 \leq s \leq \frac{\pi}{12}$}: One can check that
\[
  |\cos(s + x)| \leq 1 - \frac{1}{3} x^2 - \sin(s) \cdot x + \frac{1}{8} x^4 \quad \forall x \in \setR
\]
Then 
\begin{eqnarray*}
  \E[|\cos(s+2\pi \cdot Z)|] &\leq& 1 - \frac{(2\pi)^2}{3} \E[Z^2] - 2\pi \sin(s) \underbrace{\E[Z]}_{=0} + \frac{(2\pi)^4}{8} \E[Z^4] \\
&\leq& 1 - \underbrace{(2\pi)^2 \Big(\frac{1}{3}-\frac{(2\pi)^2}{8}\|\bm{\theta}\|_{\infty}^2-3 \cdot \frac{(2\pi)^2}{8} \underbrace{\E[Z^2]}_{\leq b}\Big)}_{\geq 1/2} \cdot \E[Z^2] \leq 1 - \frac{1}{2} \E[Z^2]
\end{eqnarray*}
using that $\|\bm{\theta}\|_{\infty} \leq \frac{1}{4}$ and $b$ is small enough.
\item \emph{Case $\frac{\pi}{20} \leq s \leq \frac{\pi}{2}$.} In this case one can verify that 
$|\cos(s+x)| \leq 0.99-\sin(s)x+100x^2$ for all $x \in \setR$. Then 
\[
  \E[|\cos(s + 2\pi Z)|] \leq 0.99 - 2\pi \sin(s) \E[Z] + 100 \cdot (2\pi)^2\underbrace{\E[Z^2]}_{\leq b} \leq 0.999
\]
for small enough $b$.
\end{itemize}
\end{proof}

The next lemma will summarize the two cases that we have distinguished so far
and show that for $\|\bm{\theta}\|_{\infty} \leq \frac{1}{4}$, every column of $\bm{A}$ reduces the value $|\hat{\bm{D}}(\bm{\theta})|$
by a factor of $1 - \Theta(\min\{ p\|\bm{\theta}\|_2^2,1\})$ in expectation. 
\begin{lemma}\label{lem:ExpectedDiscFourierDecayForOneFactor}
There is a constant $c>0$ so that the following is true:
Let $\bm{\theta} \in \setR^m$ with $\|\bm{\theta}\|_{\infty}\leq \frac{1}{4}$. 
Draw $\bm{a} \in \{ 0,1\}^m$ by letting $\Pr[a_i=1]=p$ independently for $0 \leq p \leq \frac{1}{2}$.
Then 
\[
  \E\Big[\Big|\cos\Big(2\pi \sum_{i=1}^m a_i\theta_i\Big)\Big|\Big] \leq 1 - \min\Big\{ \frac{1}{4} p\|\bm{\theta}\|_2^2,c\Big\}
\]
\end{lemma}
\begin{proof}
If $p\|\bm{\theta}\|_{\infty}^2 \geq c$, then Lemma~\ref{lem:BoundingOneFourierFactorForLargeThetaEntries}
bounds the left hand side of the claim even by $1-\frac{\pi^2}{4}c$ and we are done. Then we can assume from now on that 
$p|\theta_i|^2 \leq c$ for all $i=1,\ldots,m$ for an arbitrarily small constant $c$. 

We set $y_i := a_i-p$, so that the $y_i$'s are ``centered'' random variables in the sense that $\E[y_i]=0$. 
Note that $|y_i| \leq 1$ and $\E[y_i^2]=p\cdot (1-p)^2 + (1-p)\cdot (-p)^2 = p(1-p) \in [\frac{p}{2},p]$
as $0 \leq p \leq \frac{1}{2}$.
 
If $p\|\bm{\theta}\|_2^2 \leq b$ then we can apply Lemma~\ref{lem:ExpectedCosineForSmallL2NormCoefficients} to obtain 
\[
 \E\Big[\Big|\cos\Big(2\pi \sum_{i=1}^m a_i\theta_i\Big)\Big|\Big] = \E\Big[\Big|\cos\Big(\underbrace{2\pi \cdot p \cdot \sum_{i=1}^m \theta_i}_{=: s} + 2\pi\sum_{i=1}^m \theta_iy_i\Big)\Big|\Big] \leq 1-\frac{1}{4} p \|\bm{\theta}\|_2^2.
\]
Otherwise suppose that $p\|\bm{\theta}\|_2^2 > b$. Assuming that $c \leq b/2$, there is an index $k$
so that $\frac{b}{2} \leq p\sum_{i=1}^k |\theta_i|^2 \leq b$. Then fixing any outcome for $a_{k+1},\ldots,a_m$ we get
\begin{eqnarray*}
  \E_{a_1,\ldots,a_k}\Big[\Big|\cos\Big(2\pi \sum_{i=1}^m a_i \theta_i\Big)\Big|\Big] &=& \E_{y_{1},\ldots,y_k}\Big[\Big|\cos\Big(\underbrace{2\pi \cdot p \sum_{i=1}^m \theta_i + 2\pi\sum_{i=k+1}^m \theta_iy_i}_{=:s} + 2\pi\sum_{i=1}^k \theta_iy_i\Big)\Big|\Big] \\
&\leq& 1 - \frac{p}{2} \sum_{i=1}^k \theta_i^2 \leq 1 - \frac{1}{4} b
\end{eqnarray*}
again applying Lemma~\ref{lem:ExpectedCosineForSmallL2NormCoefficients}.
\end{proof}


The next step is to show that with high probability even the integral over the 
Fourier coefficients that are far from $\Lambda$ is tiny:
\begin{lemma} \label{lem:GeneralBoundOnSumOfFarFourierCoeff}
For parameters $ p \in [0, \frac{1}{2}]$ and $\delta>0$ so that $\frac{p\delta^2}{6} \leq 1$ and $p\delta^2\le c$, 
define 
\[
K := \left\{ \bm{\theta} \in \left[-\tfrac{1}{2},\tfrac{1}{2}\right)^m \mid d_2(\bm{\theta},\Lambda)\ge \delta \right\}.
\] 
Then with probability at least $1-\exp\left(-\frac{1}{96}p\delta^2n\right)$ one has
$\int_{\bm{\theta} \in K} |\hat{\bm{D}}(\bm{\theta})| d\bm{\theta} \leq \exp(-\frac{\delta^2}{24} pn)$.
\end{lemma}
\begin{proof}
Let us abbreviate $\phi_k(\bm{\theta}) := \prod_{j=1}^k |\cos(2\pi \sum_{i=1}^m \theta_iA_{ij})|$
and $\Phi_k := \int_{\bm{\theta} \in K} \phi_k(\bm{\theta}) d\bm{\theta}$. Intuitively, $\phi_k(\bm{\theta})$ gives the 
contribution of the first $k$ columns of the matrix $\bm{A}$ to the Fourier coefficient $|\hat{\bm{D}}(\bm{\theta})|$. In particular 
$\phi_0(\bm{\theta})=1$ and $\phi_n(\bm{\theta}) = |\hat{\bm{D}}(\bm{\theta})|$.
We have proven in Lemma~\ref{lem:ExpectedDiscFourierDecayForOneFactor} that for any fixed columns $\bm{A}^1,\ldots,\bm{A}^{k-1}$ we will have $\E[\phi_{k}(\bm{\theta})] \leq (1-\frac{p}{12}\delta^2) \cdot \phi_{k-1}(\bm{\theta})$ and hence by linearity also $\E[\Phi_k] \leq (1-\frac{p}{12}\delta^2) \cdot \Phi_{k-1}$. Let us write $\alpha_{k} \in [0,1]$ as the random variable so that $\Phi_{k} = (1 - \alpha_{k}) \cdot \Phi_{k-1}$. Then 
$\E[\alpha_k \mid \alpha_{1},\ldots,\alpha_{k-1}] \geq \frac{p\delta^2}{12}$
and $\E[\sum_{k=1}^n \alpha_k] \geq \frac{1}{12}p\delta^2n$. 
By standard Chernov/Martingale concentration bounds, we have 
\[
  \Pr\Big[\sum_{k=1}^n \alpha_k < \frac{1}{2} \E\Big[\sum_{k=1}^n\alpha_k\Big] \Big] \leq \exp\Big(-\frac{1}{8} \E\Big[\sum_{k=1}^n \alpha_k\Big]\Big) \leq \exp\Big(-\frac{1}{96}p\delta^2n\Big)
\]
Assuming that this event does not happen, we have 
\[
 \int_{\bm{\theta} \in K} |\hat{\bm{D}}(\bm{\theta})| d\bm{\theta} = \underbrace{\Phi_0}_{\leq 1} \cdot \prod_{k=1}^n (1-\alpha_k)
 \leq \exp\Big(-\sum_{k=1}^n \alpha_k\Big) \leq \exp\Big(-\frac{1}{24}p\delta^2n\Big).
\]
\end{proof}

It remains to finish the proof of Lemma \ref{lem:intDhatupperbd}. For this sake 
we can choose $\delta:=\frac{1}{16\sqrt{t}}$ and write
\[
  \int_{\bm{\theta} \in [-\frac{1}{2},\frac{1}{2})^m: d_2(\bm{\theta},\Lambda)\ge \frac{1}{16\sqrt{t}}} |\hat{\bm{D}}(\bm{\theta})| d\bm{\theta} 
\stackrel{\textrm{Lem.~\ref{lem:GeneralBoundOnSumOfFarFourierCoeff}}}{\leq} \exp\Big(-\frac{\delta^2}{24} pn\Big) = \exp\Big(- \frac{1}{(16)^2 \cdot 24} \cdot \frac{n}{m}\Big)  
\]
as claimed.

\section{Proof of the main theorem}

Finally, we can put everything together.
\begin{proof}[Proof of Theorem \ref{thm:mainthm}]
Assume that $t\geq C\log(n)$ and $n\geq Cm^2\log(m)$ for a large enough constant $C>0$. 
We can apply Lemmas \ref{lem:lowerbound}, \ref{lem:halfintegralbound} and \ref{lem:intDhatupperbd} to obtain that with high probability we have
\begin{eqnarray*}
\Pr[\bm{X}=\bm{0}]&=&\int_{[-\frac{1}{2},\frac{1}{2})^m} \hat{\bm{X}}(\bm{\theta}) d\bm{\theta}\\
&\geq& \int_{\|\bm{\theta}\|_2 \leq \frac{1}{16\sqrt{t}}} \hat{\bm{X}}(\bm{\theta}) d\bm{\theta}  - \int_{d_2(\bm{\theta},\Lambda \setminus \{\bm{0}\})\le \frac{1}{16\sqrt{t}}} \hat{\bm{X}}(\bm{\theta}) d\bm{\theta} - \underbrace{\int_{d_2(\bm{\theta},\Lambda)> \frac{1}{16\sqrt{t}}} |\hat{\bm{X}}(\bm{\theta})| d\bm{\theta}}_{\leq \exp(-\Theta(n/m))\textrm{ by Lem.}~\ref{lem:intDhatupperbd}} \\
&\ge & \frac{1}{2} \underbrace{\int_{\|\bm{\theta}\|_2 \leq \frac{1}{16\sqrt{t}}} \hat{\bm{X}}(\bm{\theta}) d\bm{\theta}}_{\geq n^{-\Theta(m)}}
-e^{-\Theta(n/m)} \stackrel{n \geq Cm^2\log(m)}{>}0.
\end{eqnarray*}
\end{proof}

\section{Alternate analysis for $\Delta=1$}
In the case of $\Delta=1$, we can simplify our analysis by choosing $\bm{R}$ in such a way that $\bm{X}=\bm{D}+\bm{R}$ will always be even. This allows us to integrate only over $[-\frac{1}{4},\frac{1}{4})^m$, so that we do not need to bound the half-integral points.
This approach is a bit simpler, but it does not seem to extend to larger $\Delta$. For this reason we give the more general version above.

More precisely, suppose we have a fixed $\bm{A}\in \{0,1\}^{m\times n}$. 
Choose $\bm{x}\sim \{-1,1\}^n$ uniformly at random and define $\bm{D}=\bm{A}\bm{x}\in \setZ^m$.
For all $i$ with $\|\bm{A}_i\|_\infty$ odd, choose $R_i\sim \{-1,1\}$ uniformly. For all other $i$, set $R_i=0$.
Then define $\bm{X}=\bm{D}+\bm{R}$.

Notice that $\bm{D}$ and $\bm{R}$ are still independent random variables and so $\hat{\bm{X}}(\bm{\theta})=\hat{\bm{D}}(\bm{\theta}) \cdot \hat{\bm{R}}(\bm{\theta})$.
Moreover, regardless of the choice of $\bm{x}$, $D_i$ has the same parity as $\|\bm{A}_i\|_1$. In particular, we have $\bm{X}\in 2\setZ^m$.

We then obtain
\[\Pr[\bm{X}=\bm{0}] 
= 2^m \int_{\bm{\theta}\in [-\frac{1}{4},\frac{1}{4})^m}\hat{\bm{X}}(\bm{\theta})d\bm{\theta}.
\]


We can compute that $\hat{\bm{R}}(\bm{\theta})=\prod_{i:\|\bm{A}_i\|_\infty \textrm{ odd}}\cos(2\pi \theta_i)$ and so $\hat{\bm{R}}(\bm{\theta})>0$ for $\bm{\theta} \in [-\frac{1}{4},\frac{1}{4})^m$.
Moreover, for $\|\bm{\theta}\|_\infty\le \frac{1}{8}$ we still have
$\hat{\bm{R}}(\bm{\theta})\ge \exp(-\pi^2\|\bm{\theta}\|_2^2-20\|\bm{\theta}\|_2^4)$.

In particular, for this choice of $\bm{R}$, Lemmas \ref{lem:lowerbound} and \ref{lem:intDhatupperbd} still hold. Since we are only integrating over $[-\frac{1}{4},\frac{1}{4})^m$, we do not need to bound the half-integral points by the decay of $\hat{\bm{R}}$, and Theorem \ref{thm:mainthm} follows.

\section{Open problems and conjectures}

As we have seen, the Fourier-analytic method works particularly well for 
random set systems. But there is no a priori reason why it could not be made to work
for arbitrary set systems. As a first step, one should wonder whether one can reprove
Spencer's theorem in our framework: 
\begin{openquestion*}
Can one use the Fourier-analytic framework to show that an arbitrary set system with $m=n$
has discrepancy $O(\sqrt{n})$?
\end{openquestion*}
On the one hand, it is clear that the Fourier 
coefficients $\hat{\bm{D}}(\bm{\theta})$ may be less well behaved than for random set systems. But on the other hand
the value of $\Delta$ can be chosen a lot larger than in our application leading to a faster
decay of $\hat{\bm{R}}(\bm{\theta})$. 

Here we give one extra remark: One could wonder, what is the random variable $R$
supported on $\{-\Delta,\ldots,\Delta\}$ whose Fourier tails decay fastest? A possible way to quantify this is 
to define $\rho(R) := \max\{ |\hat{R}(\theta)| : \frac{1}{4} \leq \theta \leq \frac{1}{2} \}$ and ask how small
$\rho(R)$ can be in terms of $\Delta$. Note that the concrete choice of the interval $[\frac{1}{4},\frac{1}{2}]$ is arbitrary
as long as it does not contain an integer. For our choice of $R \sim \pazocal{R}(\Delta)$ we have seen in \eqref{eq:OneDimRhatUpperBound} that $\rho(R) \leq \exp(-\Theta(\Delta))$. An elegant complex-analytic argument by Chris Bishop (personal communication via Yuval Peres) shows that this bound is asymptotically tight.

The next question deals with the tightness of our bound. Consider for the sake of simplicity 
the case of $p = \frac{1}{2}$ and let us remind ourselves of the
\emph{lower bound} on the discrepancy of random set systems. The argument will work up to a threshold  of $\Delta := n^{1/2-\varepsilon}$. Fix a coloring $\bm{x}$ and pick the matrix $\bm{A} \in \{ 0,1\}^{m \times n}$ at random. 
Then for any set index $i$, we will have $\Pr[|\bm{A}_i\bm{x}| \leq \Delta] \leq O(\frac{\Delta}{\sqrt{n}})$ using that the standard deviation of $\bm{A}_i\bm{x}$ is $\Theta(\sqrt{n})$. Then for that particular coloring
$\bm{x}$ we have $\Pr[\|\bm{A}\bm{x}\|_{\infty} \leq \Delta] \leq (\frac{\Delta}{\sqrt{n}})^m$. 
This bound holds for each of the $2^n$ possible colorings, and so the expected number of good colorings is bounded by $2^n \cdot (\frac{\Delta}{\sqrt{n}})^m$.
For $n\le c\varepsilon m \ln(n)$ and $0 < \Delta \leq n^{1/2-\varepsilon}$, the expected number is less than $\frac{1}{2}$, and so a good coloring cannot exist with high probability.
This argument breaks down if $n \gg \Theta(m\log(n))$. Naturally one wonders whether this is a 
tight construction and whether there is a matching upper bound: 
\begin{openquestion*}
For a large enough constant $C>0$, suppose that $n = Cm\log(m)$. Pick a matrix $\bm{A} \in \{ 0,1\}^{m \times n}$ where each entry is uniformly and independently drawn from $\{0,1\}$ (i.e. $p=\frac{1}{2}$). 
Is then $\textrm{disc}(\bm{A}) \leq n^{1/2-\varepsilon}$ with high probability?
Can one even show an optimal bound of $\disc(\bm{A}) \leq 1$ already in this regime?
\end{openquestion*}

Next, the random model where each incidence appears with probability $p$ will 
in particular create sets that have about the same size, assuming the parameters are chosen 
so that concentration effects kick in. The same holds for the model of Ezra and Lovett~\cite{BeckFialaForRandomSetSystems-EzraLovett2016}. 
Here is a more challenging \emph{semi-random} model. Let $m,n,t \in \setN$ and $0<\delta \leq 1$ be parameters. Suppose an adversary picks a \emph{distribution matrix} $\bm{P} \in [0,\delta]^{m \times n}$ 
with column sum $\|\bm{P}^j\|_1 \leq t$ for all $j \in [n]$. Then a random matrix $\bm{A} \in \{ 0,1\}^{m \times n}$
is chosen where each entry is sampled independently with $\Pr[A_{ij}=1] = P_{ij}$. 
A natural question is the following: 
\begin{openquestion*}
Suppose that $m,n \in \setN$ are arbitrary, $t \geq C \log(m+n)$ and $0<\delta \ll 1$.
Sample a random matrix $\bm{A}$ according to a distribution matrix $\bm{P} \in [0,\delta]^{m \times n}$. Can one show that 
the discrepancy of $\bm{A}$ bounded by $O(\sqrt{t} \cdot \textrm{polylog}(t))$, assuming that $\delta$ is small enough and $C>0$ is large enough?
\end{openquestion*}
Note that for $\delta = 1$, the adversary could choose a deterministic hard matrix $\bm{P} \in \{ 0,1\}^{m \times n}$ and enforce that $\bm{A}=\bm{P}$. Hence $\delta \ll 1$ would be needed and the question should be easier to answer the smaller $\delta$ is, as this adds more randomness. As an intermediate model that still allows the adversary to create sets of various sizes, one could also consider the restriction of the model where 
all entries in the same row of $\bm{P}$ are identical. 

Finally, the reader will have observed that our bound is non-constructive and we do not
know a polynomial time algorithm to find the corresponding colorings. 
\begin{openquestion*}
For $n \geq Cm^2\log(m)$ and $p \in [0,\frac{1}{2}]$ and $t=pm$ with $t \geq C\log(n)$, 
draw a random $\bm{A} \in \{ 0,1\}^{m \times n}$ with $\Pr[A_{ij}=1]=p$ independently for each entry.
Is there a polynomial time algorithm that finds a coloring $\bm{x}$ with $\|\bm{A}\bm{x}\|_{\infty} \leq 1$
with high probability?
\end{openquestion*}
In fact, also for the result of Kuperberg, Lovett and Peled~\cite{DBLP:conf/stoc/KuperbergLP12}, 
no polynomial time algorithm is known to find the constructions that are proven to exist. 
Hence answering this particular question is likely to have an impact far beyond the scope of
this paper.

\paragraph{Independent work.}
An independent work of Franks and Saks \cite{frankssaks18} uses similar techniques to show the discrepancy of random matrices in the regime $n=\tilde{\Theta}(m^3)$ is bounded by $2$ with high probability. Their work applies to a more general setting where the columns are chosen from a distribution on a lattice.

\subsection*{Acknowledgment}

The authors want to thank Chris Bishop and Yuval Peres for insights into the decay of
Fourier coefficients.

\bibliographystyle{alpha}
\bibliography{discft}

\begin{thebibliography}{BDGL17}

\bibitem[Ban98]{bana98}
Wojciech Banaszczyk.
\newblock Balancing vectors and gaussian measures of n-dimensional convex
  bodies.
\newblock {\em Random Struct. Algorithms}, 12(4):351--360, 1998.

\bibitem[Ban10]{DiscrepancyMinimization-Bansal-FOCS2010}
N.~Bansal.
\newblock Constructive algorithms for discrepancy minimization.
\newblock In {\em FOCS}, pages 3--10, 2010.

\bibitem[BCP01]{DBLP:journals/rsa/BorgsCP01}
Christian Borgs, Jennifer~T. Chayes, and Boris Pittel.
\newblock Phase transition and finite-size scaling for the integer partitioning
  problem.
\newblock {\em Random Struct. Algorithms}, 19(3-4):247--288, 2001.

\bibitem[BDG16]{AlgorithmForKomlosBansalDadushGargFOCS16}
Nikhil Bansal, Daniel Dadush, and Shashwat Garg.
\newblock An algorithm for koml{\'{o}}s conjecture matching banaszczyk's bound.
\newblock In {\em {IEEE} 57th Annual Symposium on Foundations of Computer
  Science, {FOCS} 2016, 9-11 October 2016, Hyatt Regency, New Brunswick, New
  Jersey, {USA}}, pages 788--799, 2016.

\bibitem[BDGL17]{BanaszczykAlgorithmically-BansalDadushGargLovett2017}
Nikhil Bansal, Daniel Dadush, Shashwat Garg, and Shachar Lovett.
\newblock The gram-schmidt walk: {A} cure for the banaszczyk blues.
\newblock {\em CoRR}, abs/1708.01079, 2017.

\bibitem[Bec81]{Beck-RothsEstimateIsSharp1981}
J.~Beck.
\newblock Roth's estimate of the discrepancy of integer sequences is nearly
  sharp.
\newblock {\em Combinatorica}, 1(4):319--325, 1981.

\bibitem[BF81]{BECK1981}
J~Beck and T~Fiala.
\newblock Integer making theorems.
\newblock {\em Discrete Applied Mathematics}, 3(1):1 -- 8, 1981.

\bibitem[Cha00]{Chazelle2000}
Bernard Chazelle.
\newblock {\em The Discrepancy Method: Randomness and Complexity}.
\newblock Cambridge University Press, New York, NY, USA, 2000.

\bibitem[EL75]{ErdosLovasz3ChromaticHypergraphs1975}
P.~Erdos and L.~Lov{\'a}sz.
\newblock In infinite and finite sets.
\newblock {\em Proceedings of the Colloqmium of the Math Society Janos Bolyai,
  10, Problems and results on 3-chromatic hypergraphs and some related
  questions}, pages 609--627, 1975.
\newblock cited By 1.

\bibitem[EL16]{BeckFialaForRandomSetSystems-EzraLovett2016}
Esther Ezra and Shachar Lovett.
\newblock {On the Beck-Fiala Conjecture for Random Set Systems}.
\newblock In Klaus Jansen, Claire Mathieu, Jos{\'e} D.~P. Rolim, and Chris
  Umans, editors, {\em Approximation, Randomization, and Combinatorial
  Optimization. Algorithms and Techniques (APPROX/RANDOM 2016)}, volume~60 of
  {\em Leibniz International Proceedings in Informatics (LIPIcs)}, pages
  29:1--29:10, Dagstuhl, Germany, 2016. Schloss Dagstuhl--Leibniz-Zentrum fuer
  Informatik.

\bibitem[FS18]{frankssaks18}
Cole Franks and Michael Saks.
\newblock On the discrepancy of random matrices with many columns.
\newblock arXiv 1807.04318, 2018.

\bibitem[KLP12]{DBLP:conf/stoc/KuperbergLP12}
Greg Kuperberg, Shachar Lovett, and Ron Peled.
\newblock Probabilistic existence of rigid combinatorial structures.
\newblock In {\em Proceedings of the 44th Symposium on Theory of Computing
  Conference, {STOC} 2012, New York, NY, USA, May 19 - 22, 2012}, pages
  1091--1106, 2012.

\bibitem[LM12]{DiscrepancyMinimization-LovettMekaFOCS12}
S.~Lovett and R.~Meka.
\newblock Constructive discrepancy minimization by walking on the edges.
\newblock In {\em FOCS}, pages 61--67, 2012.

\bibitem[LRR17]{DetDiscMin-LevyRamadasRothvossIPCO2017}
Avi Levy, Harishchandra Ramadas, and Thomas Rothvoss.
\newblock Deterministic discrepancy minimization via the multiplicative weight
  update method.
\newblock In Friedrich Eisenbrand and Jochen Koenemann, editors, {\em Integer
  Programming and Combinatorial Optimization}, pages 380--391, Cham, 2017.
  Springer International Publishing.

\bibitem[Mat99]{matousek1999}
J.~Matousek.
\newblock {\em Geometric Discrepancy: An Illustrated Guide}.
\newblock Algorithms and Combinatorics. Springer Berlin Heidelberg, 1999.

\bibitem[Rot14]{DiscrepancyForConvexSets-RothvossFOCS2014}
T.~Rothvoss.
\newblock Constructive discrepancy minimization for convex sets.
\newblock In {\em 2014 IEEE 55th Annual Symposium on Foundations of Computer
  Science}, pages 140--145, Oct 2014.

\bibitem[Spe85]{SixStandardDeviationsSuffice-Spencer1985}
J.~Spencer.
\newblock Six standard deviations suffice.
\newblock {\em Transactions of the American Mathematical Society},
  289(2):679--706, 1985.

\end{thebibliography}

\section{Appendix}

Here is the proof of Lemma~\ref{lem:inversionformula}:
\begin{lemma*}[Fourier Inversion Formula] 
For any integer-valued random vector $\bm{X} \in \setZ^m$ and $\bm{\lambda} \in \setZ^m$ one has
\[
 \Pr[\bm{X} = \bm{\lambda}] = \int_{[-\frac{1}{2},\frac{1}{2})^m} \hat{\bm{X}}(\bm{\theta}) \cdot \exp\big(-2\pi i \left<\bm{\lambda},\bm{\theta}\right>\big) d\bm{\theta} 
\]
\end{lemma*} 
\begin{proof}
We have
\begin{eqnarray*}
 \int_{[-\frac{1}{2},\frac{1}{2})^m} \hat{\bm{X}}(\bm{\theta}) \cdot \exp\big(-2\pi i \left<\bm{\lambda},\bm{\theta}\right>\big) d\bm{\theta} &=& \int_{[-\frac{1}{2},\frac{1}{2})^m} \E\big[\exp\big(2\pi i \left<\bm{X}-\bm{\lambda},\bm{\theta}\right>\big)\big] d\bm{\theta} \\
 &=& \E\Big[\underbrace{\int_{[-\frac{1}{2},\frac{1}{2})^m} \exp\big(2\pi i \left<\bm{X} - \bm{\lambda}, \bm{\theta}\right>\big) d\bm{\theta}}_{=1\textrm{ if }\bm{X}-\bm{\lambda}=\bm{0}, =0\textrm{ otherwise}}\Big] \\
&=& \Pr[\bm{X}-\bm{\lambda} = \bm{0}]
\end{eqnarray*}
Here in the last step, we have used a cancellation that we prove in more
detail in the next lemma. 
\end{proof}
\begin{lemma}
Let $D := [-\frac{1}{2},\frac{1}{2})^m$ and $\bm{t} \in \setZ^m$. Then
\[
  \int_D \exp(2\pi i\left<\bm{t},\bm{\theta}\right>) d\bm{\theta} = \begin{cases} 1 & \textrm{if }\bm{t} \neq \bm{0} \\ 0 & \textrm{if } \bm{t}=\bm{0}. \end{cases}
\]
\end{lemma}
\begin{proof}
If $\bm{t}=\bm{0}$, then $\int_D \exp(0) d\bm{\theta} = \textrm{vol}_n(D) = 1$. Now suppose that $\bm{t} \neq \bm{0}$. W.l.o.g. assume that $t_1 \neq 0$. We split
the integral as 
\begin{eqnarray*}
 & &  \int_{D} \exp(2\pi i\left<\bm{t},\bm{\theta}\right>) d\bm{\theta} \\
&=& \int_{\theta_2,\ldots,\theta_n \in [-\frac{1}{2},\frac{1}{2})} \exp\Big(2\pi i \sum_{j=2}^n \theta_jt_j\Big) \cdot \Big( \underbrace{\int_{-1/2}^{1/2} \exp(2\pi i t_1\theta_1) d\theta_1}_{=0 \; (*)} \Big) d\theta_2,\ldots,\theta_n = 0
\end{eqnarray*}
Here $(*)$ is true since the integral starts and ends at the same point $\exp(-\pi i t_1) = \exp(+\pi i t_1)$
and goes through the complex unit circle $t_1$ times. Hence all values must cancel out.
\end{proof}

\end{document}